\documentclass[14pt]{article}
\usepackage[dvips]{graphicx}
\usepackage{amsfonts}
\usepackage{tikz}
\usetikzlibrary{arrows,decorations.pathmorphing}
\usetikzlibrary{backgrounds,positioning,fit,petri}
\usetikzlibrary{matrix}
\usepackage{amsmath}
\usepackage{amsfonts}
\usepackage{amssymb,amsthm}
\newcommand{\calH}{{\mathcal H}}

\newcommand{\N}{{\mathbb N}}
\renewcommand{\P}{\N_{>0}}

\newcommand{\Q}{{\mathbb Q}}

\newcommand{\Frac}[2]{\displaystyle \frac{#1}{#2}}

\def\Lyn{{\mathcal Lyn}}

 \def\shuffle{\mathop{_{^{\sqcup\!\sqcup}}}} 

\gdef\stuffle{\;% 
  \setlength{\unitlength}{0.0125cm}% 
  \begin{picture}(20,10)(220,580) 
  \thinlines 
  \put(220,592){\line( 0,-1){ 10}} 
  \put(220,582){\line( 1, 0){ 20}} 
  \put(240,582){\line( 0, 1){ 10}} 
  \put(230,592){\line( 0,-1){ 10}} 
  \put(225,587){\line( 1, 0){ 10}} 
  \end{picture}\; 
}

\newtheorem{proposition}{Proposition}
\newtheorem{theorem}{Theorem}
\newtheorem{lemma}{Lemma}
\newtheorem{definition}{Definition}

\newtheorem{remark}{Remark}

\def\P{\mathrm{P}}

\def\A{\mathrm{A}}

\def\deg{\mathrm{deg}}

\def\QX{\poly{\Q}{X}}

\def\Lie{{\cal L}ie}

\def\LQX{\Lie_{\Q} \langle X \rangle}

\def\Frac{\displaystyle\frac}
\def\path{\rightsquigarrow}

\def\bv{\mid}

\gdef\minishuffle{{\scriptstyle \shuffle}}  
\gdef\ministuffle{{\scriptstyle \stuffle}}

\def\deg{\mathop\mathrm{deg}\nolimits}

\usepackage{color}
\usepackage{ulem}
\definecolor{MyDarkBlue}{rgb}{0,0.08,0.4} 
 
\newcommand{\mref}[1]{eq.\ref{#1} }
\def\ra{\rightarrow}
\def\lra{\longrightarrow}

\def\calD{{\cal D}}

\def\Tens{\mathrm{T}}
\def\ep{\epsilon}

\def\ul#1{\underline{#1}}

\def\AY{A \langle Y \rangle}
\def\QX{\Q\langle X \rangle}

\def\End{\mathrm{End}}
\def\Hom{\mathrm{Hom}}
\def\Im{\mathrm{Im}}
\def\Prim{\mathrm{Prim}}
\def\conc{\mathtt{conc}}
\def\scal#1#2{\langle #1\bv#2 \rangle}
\def\ncp#1#2{#1\langle #2\rangle}
\def\ncs#1#2{#1\langle \!\langle #2\rangle \!\rangle}
\def\binomial#1#2{\left(\,\begin{matrix}#1 \\ #2\end{matrix}\,\right)}
\def\bv{\mid}

\def\LQX{\Lie_{\Q} \langle X \rangle}

\def\A{\mathcal{A}}
\def\B{\mathcal{B}}
\def\U{\mathcal{U}}
\def\SG{{\mathfrak S}}
\def\ot{\otimes}
\def\up#1{\raise 1ex\hbox{\footnotesize#1}}
\def\pointir{\unskip . --- \ignorespaces}
\newcounter{per1}
\begingroup
\def\2#1{\ifnum#1<10 0\fi\the#1}
\xdef\isodayandtime{
{\2\day-\2\month-\the\year\space\2{\count0}:%
\2{\count2}}}
\endgroup

\begin{document}

\begin{center}
{\Large Combinatorics of deformed shuffle 
Hopf algebras\footnote{Version du \isodayandtime}}

\bigskip
{\Large G\' erard H. E. Duchamp$^{1\,2}$}

{\Large Hoang Ngoc Minh$^{2\,3}$}

{\Large Christophe Tollu $^{1\,2}$}

{\Large Chi\^en B\`ui $^{4\,2}$}

{\Large Hoang Nghia Nguyen$^{1\,2}$}

\medskip \noindent $^1$Universit\'e Paris 13, 99, avenue Jean-Baptiste Cl\'ement, 93430 Villetaneuse, France.\\
\medskip \noindent $^2$LIPN - UMR 7030, CNRS, 93430 Villetaneuse, France.\\
\medskip \noindent $^3$Universit\'e Lille II, 1, Place D\'eliot, 59024 Lille, France.\\
\end{center}

\begin{abstract}
In order to extend Sch\"utzenberger's factorization to general perturbations, the combinatorial aspects of the Hopf algebra of a deformed shuffle product is developed systematically in a parallel way with those of the shuffle product, with an emphasis on the Lie elements as studied by Ree. In particular, we will give an effective construction of pair of bases in duality.
\end{abstract}

\tableofcontents

\section{Introduction}

Many algebras of functions \cite{DDHS} and many special sums \cite{SLC43,SLC44} are governed by shuffle products, their perturbations (adding a ``superposition term'' \cite{SLC62}) or deformations \cite{th_ung}.\\
In order to better understand the mechanisms of these products, we wish here to examine, with full generality the products which are defined by a recursion of the type \cite{orsay}
\begin{eqnarray}
au \star bv&=&a\,(u\star bv)+b\,(au\star v)+\phi(a,b)\, u \star v\ ,
\end{eqnarray}
the empty word being the neutral of this new product.\\
We then provide some  classical combinatorial applications. In most cases, the law $\phi$ is dual\footnote{That is to say comes by dualization of a comultiplication.} and under some growth conditions the obtained algebra is an enveloping algebra. 

\smallskip    
%The structure of the paper is the following {\tt TODO ...}.\\
In the second section, there is a version of the Cartier-Quillen-Milnor and Moore \footnote{CQMM in the sequel} without any use of the Poincar\'e-Birkhoff-Witt construction. We are obliged to restate the CQMM theorem without supposing any basis because we aim at ``varying the scalars'' in forthcoming papers (germs of functions, arithmetic functions, etc.) and, in order to do this at ease, we must cope safely with cases where torsion (non-zero annihilators) may appear (and then, one cannot have any basis). See (counter) examples in the section.  

\smallskip
{\bf Acknowledgements}\pointir The authors wish to thank Darij Grinberg for having thoroughly read the manuscript for having thoroughly read the manuscript, provided a limiting counterexample and participated to fruitful interactions. The authors also would like to acknowledge the support of the ``Combinatoire alg\'ebrique'' Univ. Paris 13, Sorbonne Paris Cit\'e BQR grant.

\section{First steps}

Let $X$ be a totally ordered alphabet\footnote{In the sequel, the order between the words will be understood as the lexicographic total ordering $<$. For example, with $a<b$, one has 
$ab< b$.}. The free monoid and the set of Lyndon words, over $X$, are denoted respectively by $X^*$ and $\Lyn X$. The neutral element of $X^*$, {\it i.e.} the empty word is denoted by $1_{X^*}$. 
Let $\ncp{\Q}{X}$ be equipped with the concatenation and the shuffle product which is defined on the words by 
\begin{eqnarray}\label{shuff_rec}
\forall w\in X^*,&&w\shuffle 1_{X^*}=1_{X^*}\shuffle w=w,\cr
\forall x,y\in X,\forall u,v\in X^*,&&xu\shuffle yv=x(u\shuffle yv)+y(xu\shuffle v),
\end{eqnarray}
or by their dual co-products, $\Delta=\Delta_{\tt conc}$ and $\Delta=\Delta_{\shuffle}$, defined, for any $w\in X^*$ by,
\begin{eqnarray}
\Delta_{\tt conc}(w)&=&\sum_{w=uv} u\otimes v\cr 
\Delta_{\shuffle}(w)&=&\sum_{I+J=[1..|w|]} w[I]\otimes w[J]
\end{eqnarray}
One gets two Hopf algebras
\begin{eqnarray}
\calH_{\shuffle}=(\QX,{\tt conc},1_{X^*},\Delta_{\shuffle},\epsilon,a_{\bullet})\ \mathrm{and} \cr
\calH_{\shuffle}^{\vee}=(\QX,\shuffle,1_{X^*},\Delta_{\tt conc},\epsilon,a_{\shuffle})	
\end{eqnarray}
mutually dual with respect to the pairing given by 
\begin{equation}\label{pairing}
	(\forall u,v\in X^*)(\scal{u}{v}=\delta_{u,v})\ . 
\end{equation}
The antipodes and the co-units are the same and given by, for $x_{i_1},\ldots ,x_{i_r}\in X$ and $P\in \ncp{\Q}{X}$,
\begin{eqnarray}
&&\epsilon(P)=\scal{P}{1_{X^*}},\label{counit}\cr
&&a_{\shuffle}(w)=a_{\bullet}(w)=(-1)^{r}x_{i_r}\ldots x_{i_1},\ .
\end{eqnarray}

By the CQMM theorem, the connected, graded positively, cocommutative Hopf algebra $\calH_{\shuffle}$ is isomorphic to the enveloping algebra of the Lie algebra of its primitive elements which here is $\LQX$.
Hence any  basis of the free algebra $\LQX$\footnote{The basis can be reindexed by Lyndon words and then one uses the canonical factorization of the words.} can be completed, by the PBW construction, as a linear basis $\{b_w\}_{w\in X^*}$ of ${\cal U}(\LQX)=\ncp{\Q}{X}$ (see below (\ref{PBW_shuff_basis}) for an example of such a construction) and, when the basis is finely homogeneous, so is $\{b_w\}_{w\in X^*}$ and one can construct, by duality, a basis $\{\check b_w\}_{w\in X^*}$ of  $\cal H_{\shuffle}$ (viewed as a $\Q$-module) such that~:
\begin{eqnarray}\label{prodscal}
\forall u,v\in X^*,\quad\scal{\check b_u}{b_v}&=&\delta_{u,v}\ .
\end{eqnarray}
For $w=l_1^{i_1}\ldots l_k^{i_k}$ with $l_1,\ldots l_k\in\Lyn X,\ l_1>\ldots> l_k$  
\begin{eqnarray}
\check b_{w}&=&\Frac{\check b_{l_1}^{\shuffle i_1}\shuffle\ldots\shuffle \check b_{l_k}^{\shuffle i_k}}{i_1!\ldots i_k!}.
\end{eqnarray}
(see \cite{BDKMT,acta,VJM}).
For example, Chen, Fox and Lyndon \cite{lyndon} constructed the PBW-Lyndon basis $\{P_w\}_{w\in X^*}$ for ${\cal U}(\LQX)$ as follows
\begin{eqnarray}\label{PBW_shuff_basis}
P_x=&x& \mbox{for }x\in X,\cr
P_{l}=&[P_s,P_r]&\mbox{for }l\in\Lyn X,\mbox{with standard factorization } l=(s,r),\label{recurrence}\cr
P_{w}=&P_{l_1}^{i_1}\ldots P_{l_k}^{i_k}&\mbox{for }w=l_1^{i_1}\ldots l_k^{i_k},\ l_1>\ldots>l_k,l_1\ldots,l_k\in\Lyn X.
\end{eqnarray}
Sch\"utzenberger and his school constructed the linear basis $\{S_w\}_{w\in X^*}$ for\\ ${\cal A}=(\QX,\shuffle,1_{X^*})$ by duality (w.r.t. \mref{pairing})
and obtained the transcendence basis of ${\cal A}$, $\{S_l\}_{l\in\Lyn X}$ as follows\footnote{Therefore ${\cal A}$ is a polynomial algebra ${\cal A}\simeq \Q[\Lyn X]$.}
\begin{eqnarray}
S_l=&xS_u,&\mbox{for }l=xu\in\Lyn X,\\
S_w=&
\Frac{S_{l_1}^{\shuffle i_1}\shuffle\ldots\shuffle S_{l_k}^{\shuffle i_k}}{i_1!\ldots i_k!}
&\mbox{for }w=l_1^{i_1}\ldots l_k^{i_k},l_1>\ldots>l_k.
\end{eqnarray}
After that, M\'elan\c{c}on and Reutenauer \cite{reutenauer} proved that\footnote{\label{preserves}
Recall that the duality preserves the (multi)homogeneous degrees and interchanges the triangularity of polynomials \cite{reutenauer}.
For that, one can construct the triangular matrices $M$ and $N$ whose entries are the coefficients of  the multihomogeneous triangular polynomials,
$\{\P_w\}_{w\in X^k}$ and $\{S_w\}_{w\in X^k}$ in the basis $\{w\}_{w\in X^*}$, respectively~:
\begin{eqnarray*}
M_{u,v}=\scal{P_u}{v}&\mbox{and}&N_{u,v}=\scal{S_u}{v}.
\end{eqnarray*}
The triangular matrices $M$ and $N$ are unipotent and satisfy the identity $N=({}^tM)^{-1}$. In Eq. \ref{basesdualles}, the underlined words $\ul{u}$ stand for their multidegree i.e. 
$$
\ul{u}=(|u|_x)_{x\in X}
$$}, for any $w\in X^*$,
\begin{eqnarray}\label{basesdualles}
P_w= w+\sum_{v>w,\ul{v}=\ul{w}}c_vv&\mbox{and}&S_w= w+\sum_{v<w,\ul{v}=\ul{w}}d_vv.
\end{eqnarray}
On other words, the elements of the bases $\{S_w\}_{w\in X^*}$ and $\{P_w\}_{w\in X^*}$ are upper and lower triangular respectively and are multihomogeneous.

Moreover, thanks to the duality of the bases $\{P_w\}_{w\in X^k}$ and $\{S_w\}_{w\in X^k}$, if $\calD_X$ denotes the diagonal series over $X$ one has
\begin{eqnarray}\label{factorisation}
\calD_X=\sum_{w\in X^*}w\otimes w=\sum_{w\in X^*}S_w\otimes P_w=\prod_{l\in\Lyn X}^{\searrow}\exp(S_l\otimes P_l).
\end{eqnarray}
In fact as stated in \cite{reutenauer}, this factorization holds in the framework of enveloping algebras and it will be shown in detail how to handle this framework even in the absence of any basis. It is CQMM with an analytic point of view.\\

\section{General results on summability and duality}
\subsection{Total algebras and duality}
\subsubsection{Series and infinite sums}

We here recall the results used to handle infinite sums in the sequel. The underlying topology is that of the pointwise convergence (the target being undowed with the discrete topology). This section may therefore be skipped by the reader which is familiar with these matters.\\ 
In the sequel, we will need to construct spaces of functions on different monoids (mainly direct products of free monoids). We set, once for all the general construction of the corresponding convolution algebra.\\
Let $A$ be a unitary commutative ring and $M$ a monoid. Let us denote $A^M$ the set\footnote{In general $Y^X$ is the set of all (total) mappings $X\ra Y$ \cite{B_E} Ch 2.5.2.} of all (graphs of) mappings $M\ra A$. This set is endowed with its classical structure of module. 
In order to extend the product defined in $A[M]$ (the algebra of the monoid $M$), it is essential that, in the sums 
\begin{eqnarray}\label{convol1}
	f\ast g=\sum_{m\in M}\Big(\sum_{uv=m} f(u)g(v)\Big)m
\end{eqnarray}
 the inner sums $\sum_{uv=m} f(u)g(v)$ make sense. For that, we suppose that the monoid $M$ fulfills condition ``D'' (i.e. $M$ is of finite decomposition type \cite{B_alg_I_III} Ch III.10). Formally, we say that $M$ satisfies condition ``D'' iff, for all $m\in M$, the set 
\begin{equation}
\{(u,v)\in M\times M\bv uv=m\}	
\end{equation}
is finite. In this case \mref{convol1} endows $A^M$ with the structure of an  AAU\footnote{Associative Algebra with Unit.}. This algebra is traditionally called the total algebra of $M$ (see \cite{B_alg_I_III} Ch III.10) and has very much to do with the algebra of  series\footnote{Actually, the algebra of commutative (resp. noncommutative) series on an alphabet $X$ is the total algebra of the free commutative (resp. free) monoid on $X$}. Here, it will be denoted, with an unambiguous abuse of denotation, by $\ncs{A}{M}$.

The pairing 
\begin{eqnarray}
\ncs{A}{M}\otimes A[M]&\lra&A
\end{eqnarray}
defined by\footnote{Here $A[M]$ is identified with the submodule of finitely supported functions $M\ra A$.}
\begin{eqnarray}
	\scal{f}{g}&:=&\sum_{m\in M}f(m)g(m)
\end{eqnarray}
allows to consider the total algebra as the dual of the module $A[M]$ i.e., through this pairing 
$$
\ncs{A}{M}\simeq (A[M])^*\ .
$$

One says that a family $(f_i)_{i\in I}$ of $\ncs{A}{M}$ is summable \cite{berstel_reutenauer} iff, for every $m\in M$, the mapping $i\mapsto \scal{f_i}{m}$ is finitely supported. In this case, the sum $\sum_{i\in I}f_i$ is exactly the mapping $m\longmapsto \sum_{i\in I}\scal{f_i}{m}$ so that, one has by definition
\begin{equation}
	\scal{\sum_{i\in I}f_i}{m}=\sum_{i\in I}\scal{f_i}{m}\ .
\end{equation}
Finally, let us remark that the set $M_1\otimes M_2=\{u\otimes v\}_{(u,v)\in M_1\times M_2}$ is a (monoidal) basis of $A[M_1]\otimes A[M_2]$ and $M_1\otimes M_2$ is a monoid (in the product algebra $A[M_1]\otimes A[M_2]$) isomorphic to the direct product $M_1\times M_2$.  

\subsubsection{Summable families in $\Hom$ spaces.}\label{summable}
In fact, $\ncs{A}{M}\simeq (A[M])^*=\Hom(A[M],A)$ and the notion of summability developed above can be seen as a particular case of that of a family of endomorphisms $f_i\in \Hom(V,W)$ for which $\Hom(V,W)$ appears as a complete space. It is indeed the pointwise convergence for the discrete topology. We will not expand that topic here.

The definition is similar of that of a summable family of series \cite{berstel_reutenauer}, viewed as a family of linear forms. 

\begin{definition} i) A family $(f_i)_{i\in I}$ of elements in $\Hom(V,W)$ is said to be {\rm summable} iff for all $x\in V$, the map $i\mapsto f_i(x)$ has finite support. As a quantized criterium it reads 
\begin{equation}\label{summability_criterium}
(\forall x\in V)(\exists F\subset I,\ F \mathit{finite})(\forall i\notin F)(f_i(x)=0)	
\end{equation}
ii) If the family $(f_i)_{i\in I}\in\Hom(V,W)^I$ fulfils the condition \ref{summability_criterium} above its sum is given by
\begin{equation}\label{infinite_sum}
	(\sum_{i\in I}f_i)(x)=\sum_{i\in I}f_i(x)
\end{equation}

\end{definition}
It is an easy exercise to show that the mapping $V\ra W$ defined by the equation \ref{infinite_sum} is in fact in $\Hom(V,W)$. Remark that, as the limiting process is defined by linear conditions, if a family $(f_i)_{i\in I}$ is summable, so is 
\begin{equation}\label{scaling}
	(a_if_i)_{i\in I}
\end{equation}
for an arbitrary family of coefficients $(a_i)_{i\in I}\in A^I$.\\ 
This tool will be used in section (\ref{CCQMM}) to give an analytic presentation of the theorem of Cartier-Quillen-Milnor-Moore in the case when $V=W=\B$ is a bialgebra.

\smallskip
The most interesting feature of this operation is the interchange of sums. Let us state it formally as a proposition the proof of which is left to the reader. 
\begin{proposition}\label{partition_of_indices} Let $(f_i)_{i\in I}$ be a family of elements in $\Hom(V,W)$ and $(I_j)_{j\in J}$ be a partition of $I$ (\cite{B_E} ch II \S 4 n\up{o} 7 Def. 6), then, the following statements are equivalent\\
i) $(f_i)_{i\in I}$ is summable\\
ii) for all $j\in J$, $(f_i)_{i\in I_j}$ is summable and the family $(\sum_{i\in I_j}f_i)_{j\in J}$ is summable.\\
In these conditions, one has
\begin{equation}
\sum_{i\in I}f_i=\sum_{j\in J}(\sum_{i\in I_j}f_i)	
\end{equation}  
\end{proposition} 
We derive at once from this the following practical criterium for double sums. 
\begin{proposition}\label{double_sums} Let $(f_{\alpha,\beta})_{(\alpha,\beta)\in X\times Y}$ be a doubly indexed summable family in $Hom(V,W)$, then, for fixed $\alpha$ (resp. $\beta$) the ``row-families'' $(f_{\alpha,\beta})_{\beta\in Y}$ (resp. the ``column-families'' $(f_{\alpha,\beta})_{\alpha\in X}$) are summable and their sums are summable. Moreover
\begin{equation}
\sum_{(\alpha,\beta)\in X\times Y}f_{\alpha,\beta}=
\sum_{\alpha\in X}\sum_{\beta\in Y}f_{\alpha,\beta}=
\sum_{\beta\in Y}\sum_{\alpha\in X}f_{\alpha,\beta}\ . 
\end{equation}
\end{proposition}

\subsubsection{Substitutions}
Let $\A$ be an AAU and $f\in\A$. For every polynomial $P\in \ncp{A}{X}$ ($=A[X^*]$, one can compute $P(f)$ by 
\begin{equation}
	P(f)=\sum_{n\geq 0}\scal{P}{X^n}f^n\ .
\end{equation}
One checks at once that $P\mapsto P(f)$ is a morphism\footnote{In case $\A$ is a geometric space, this morphism is called ``evaluation at $f$'' and corresponds to a Dirac measure.} of AAU's between $\ncp{A}{X}$ and $\A$. Moreover, this morphism is compatible with the substitutions as one checks easily that, for $Q\in A[X]$ 
\begin{equation}
P(Q)(f)=P(Q(f))	
\end{equation}
(it suffices to check that $P\mapsto P(Q)(f)$ and $P\mapsto P(Q(f))$ are two morphisms which coincide on
 $P=X$).\\
In order to substitute within series, one needs some limiting process. The framework of $\A=\Hom(V,W)$ and summable families will be here sufficient (see paragraph \ref{summable}). We suppose that $(V,\delta_V,\ep_V)$ is a co-AAU and that $(W,\mu_W,1_W)$ is an AAU. Then $(\Hom(V,W),*,e)$ is an AAU (with $e=1_W\circ \ep_V)$. A series $S\in A[[X]]$ and $f\in \Hom(V,W)$ being given, we say that $f\in  Dom(S)$ iff the family $(\scal{S}{X^n}f^{*n})_{n\geq 0}$ is summable\footnote{Where $f^{*n}$ denotes straightforwardly the $n$-th power of $f$ w.r.t. the convolution product.}. We have the following properties
\begin{proposition}
If $f\in Dom(S)\cap Dom(T)$ and $\alpha\in A$, one has
\begin{equation}\label{morph_lin}
(\alpha S)(f)=\alpha S(f)\ ;\	(S+T)(f)=S(f)+T(f)
\end{equation}
and
\begin{equation}\label{morph_mult}
(TS)(f)=T(f)*S(f)\ .
\end{equation}
If $((f)^{*n})_{n\geq 0}$ is summable and $S(0)=0$ then 
\begin{equation}\label{domains}
f\in Dom(S)\cap Dom(T(S))\ ;\ S(f)\in Dom(T)
\end{equation}
and
\begin{equation}\label{subs}
	T(S)(f)=T(S(f))
\end{equation}
\end{proposition}

\begin{proof}
Let us first prove \mref{morph_mult}. As
$f\in Dom(S)\cap Dom(T)$,\\ 
the families $(\scal{S}{X^n}f^{*n})_{n\geq 0}$ and $(\scal{T}{X^m}f^{*m})_{n\geq 0}$ are summable, then so is
\begin{equation}
\Big(\scal{T}{X^m}f^{*m}*\scal{S}{X^n}f^{*n}\Big)_{n,m\geq 0}
\end{equation}
as, for every $x\in V$, $\delta(x)=\sum_{i=1}^N x_i^{(1)}\ot x_i^{(2)}$ and for every $i\in I$, 
$$
supp_{w.r.t.\ m}(\scal{T}{X^m}f^{*m}(x_i^{(1)}))\ ;\ supp_{w.r.t.\ n}(\scal{S}{X^n}f^{*n}(x_i^{(2)}))
$$ 
are finite. Then outside of the cartesian product of the (finite) union of these supports, the product 
\begin{equation}
(\scal{T}{X^m}f^{*m}*\scal{S}{X^n}f^{*n})(x)=
\mu_W((\scal{T}{X^m}f^{*m}\ot\scal{S}{X^n}f^{*n})(\delta(x)))
\end{equation}
is zero. Hence the summability.\\ 
Now
\begin{eqnarray}
T(f)*S(f)&=&\sum_{m=0}^\infty(\scal{T}{X^m}f^{*m})*\sum_{n=0}^\infty(\scal{S}{X^n}f^{*n})\cr
         &=&\sum_{m=0}^\infty\sum_{n=0}^\infty (\scal{T}{X^m}\scal{S}{X^n}f^{*n+m})\cr
         &=&\sum_{s=0}^\infty\Big(\sum_{n+m=s}^\infty \scal{T}{X^m}\scal{S}{X^n}\Big)f^{*s}\cr
         &=&\sum_{s=0}^\infty (\scal{TS}{X^s})f^{*s}=(TS)(f)
\end{eqnarray}
We now prove the statements (\ref{domains}) and (\ref{subs}). If $((f)^{*n})_{n\geq 0}$ is summable
then $f$ belongs to all domains (i.e. is universally substitutable) by virtue of \mref{scaling}. For all $x\in V$, there exists $N_x\in \N$ such that 
$$
n>N_x\Longrightarrow (f)^{*n}(x)=0\ .
$$
Now, for $S$ such that $S(0)=0$, one has $S=\sum_{n=1}^\infty \scal{S}{X^n}X^n$ and then $S^k=\sum_{n=k}^\infty \scal{S^k}{X^n}X^n$. Now, in view of \mref{morph_mult}, one has
\begin{equation}
	S(f)^{*n}(x)=S^n(f)(x)=\sum_{m=n}^\infty \scal{S^n}{X^m}(f)^{*m}(x)
\end{equation}
which is zero for $n>N_x$. Hence the summability of $(S(f)^{*n})_{n\geq 0}$ which implies that $S(f)\in Dom(T)$. 
The family $(\scal{T}{X^n}\scal{S^n}{X^m}(f)^{*m})_{(n,m)\in \N^2}$ is summable because, if $x\in V$ and if $n$ or $m$ is greater than $N_x$ then 
\begin{equation}
	\scal{T}{X^n}\scal{S^n}{X^m}(f)^{*m}(x)=0
\end{equation}
 thus $T(S(f))$ is then computed by (where we use the fact that, if $S(0)=0$, then $\scal{S^n}{X^m}=0$ for $m<n$)
\begin{eqnarray}
T(S(f))	&=&\sum_{n=0}^\infty \scal{T}{X^n}S(f)^{*n}=	
	\sum_{n=0}^\infty \scal{T}{X^n}\Big(\sum_{m=n}^\infty \scal{S^n}{X^m}(f)^{*m}\Big)\cr
	&=&\sum_{n=0}^\infty\sum_{m=0}^\infty \scal{T}{X^n} \scal{S^n}{X^m}(f)^{*m}= 
	\sum_{m=0}^\infty \Big(\sum_{n=0}^\infty \scal{T}{X^n}\scal{S^n}{X^m}\Big)(f)^{*m}\cr
	&=&\sum_{m=0}^\infty \scal{T(S)}{X^m}(f)^{*m}=T(S)(f)\ .
\end{eqnarray}
\end{proof}  

In the free case (i.e. $V=W$ are the bialgebra $(\ncp{A}{X},\conc,1_{X^*},\Delta_{\shuffle},\epsilon)$), one has a very useful representation of the convolution algebra $\Hom(V,W)$ through images of the diagonal series. This representation will provide us with the key lemma (\ref{key_for_ortho}). Let  
\begin{eqnarray*}
\calD_X&=&\sum_{w\in X^*}w\otimes w.
\end{eqnarray*}
be the diagonal series attached to $X$.  
\begin{proposition}\label{diagonal_rep}
 Let $A$ be a commutative unitary ring and $X$ an alphabet. Then\\
\begin{enumerate}
	\item For every $f\in \End(\ncp{A}{X})$, the family $(u\otimes f(u))_{u\in X^*}$ is summable in $\ncs{A}{X^*\ot X^*}$.  
	\item The representation 
\begin{equation}
f\mapsto \rho(f)=\sum_{u\in X^*}u\otimes f(u)
\end{equation} 
is faithful from
$(\End(\ncp{A}{X}),*)$	to $(\ncs{A}{X^*\ot X^*},\shuffle\ot\conc)$. In particular, for 
$f\in \End(\ncp{A}{X})$ and $P\in A[X]$, one has
\begin{equation}
	\rho(P(f))=P(\rho(f))
\end{equation}
\item If $f(1_{X^*})=0$ and $S\in A[[X]]$ is a series, then $(\rho(f)^n)_{n\geq 0}$ is summable in $(\ncs{A}{X^*\ot X^*},\shuffle\ot\conc)$ and
\begin{equation}
	\rho(S(f))=S(\rho(f))
\end{equation}
\end{enumerate}
\end{proposition}
\begin{proof}(of Prop.(\ref{diagonal_rep})) (i) and (iii) are easily checked. For (ii), let us compute
\begin{eqnarray}
&& \rho(f)(\shuffle\ot\conc)\rho(g)=
\sum_{u,v\in X^*}(u\otimes f(u))(\shuffle\ot\conc)(v\otimes g(v))\cr
&=&\sum_{u,v\in X^*}(u\shuffle v)\ot (\conc(f(u)\ot g(v)))\cr
&=&\sum_{u,v\in X^*}\sum_{w\in X^*}(\scal{u\shuffle v}{w}w\ot \conc(f(u)\ot g(v))\cr
&=&\sum_{w\in X^*} w\ot \Big(\sum_{u,v\in X^*}(\scal{u\shuffle v}{w}\conc(f(u)\ot g(v))\Big)\cr
&=&\sum_{w\in X^*} w\ot \Big(\sum_{u,v\in X^*}(\scal{u\ot v}{\Delta(w)}\conc(f(u)\ot g(v))\Big)\cr
&=&\sum_{w\in X^*} w\ot (\conc\circ (f\ot g)\circ\Delta)[w]=\sum_{w\in X^*} w\ot (f*g)[w]
\end{eqnarray}
Moreover, $\rho$ is faithful because 
($\rho(f)=0\Longrightarrow f=0$).
\end{proof}  
\subsection{Theorem of Cartier-Quillen-Milnor-Moore (analytic form)}\label{CCQMM}
\subsubsection{General properties of bialgebras} 
{\bf From now on, we suppose that $A$ be a unitary commutative $\Q$-algebra (i.e. $\Q\subset A$).}

The aim of Cartier-Quillen-Milnor-Moore theorem is to provide necessary and sufficient conditions for $\B$ to be an enveloping algebra, we will discuss this condition in detail in the sequel.  

\smallskip
Let $(\B,\mu,e_\B,\Delta,\epsilon)$ be a (general) $A$-bialgebra. One can always consider the Lie algebra of primitive elements $Prim(\B)$ and build the map 
$$
j_\B : \U(Prim(\B))\ra \B\ .
$$
Then, $\A=j_\B(\U(Prim(\B)))$ is the subalgebra generated by the primitive elements. 
\begin{figure}[h]
\centering
\begin{tikzpicture}
  \matrix (m) [matrix of math nodes,row sep=3em,column sep=8em,minimum width=2em]
  {\Prim(\B) & \A & \B\\
\U(\Prim(\B))&       &\\
};
\path[right hook->] 
    (m-1-1) edge [] node [above] {$i_{\A,P}$} (m-1-2)
            edge [] node [left] {$i_{\U,P}$} (m-2-1)
    (m-1-2) edge [] node [above] {$i_{\B,\A}$} (m-1-3); 
\path[-stealth]             
    (m-2-1) edge node [above] {$i_{\A,\U}$}  (m-1-2)
    (m-2-1) edge node [below] {$j_{\B}$}  (m-1-3);
\end{tikzpicture}
\caption{The sub-algebra $\A$ generated by primitive elements.}\label{primitive_alg}
\end{figure}
 
\smallskip
The mapping $i_{\B,\A}$ is into but $i_{\B,\A}\otimes i_{\B,\A}$ may not be so. This is the case for $\B=(\Q[\ep][x],.,1_{\Q[\ep][x]},\Delta,c)$ where $(\Q[\ep][x],.,1_\B)$ is the usual polynomial algebra with coefficients in the algebra of dual numbers $\Q[\ep]$ (with $\ep^2=0$) and  
$$
\Delta(x)=x\ot 1+1\ot x+ \ep x\ot x,\ c(x)=0
$$ 
(see details and proofs below, in sec. \ref{cex}).\\ 
In general, one has (only) $\Delta_\B(\A)\subset Im(i_{\B,\A}\otimes i_{\B,\A})$, this can be simply seen from the following combinatorial argument.\\ 
For any list of primitive elements $L=[g_1,g_2,\cdots ,g_n]$ and\\ 
$I=\{i_1<i_2<..<i_k\}\subset \{1,2,..,n\}$, put $L[I]=g_{i_1}g_{i_1}\cdots g_{i_k}$, the product of the sublist.
One has 
\begin{equation}\label{subwords_of_primitive}
\Delta(g_1g_2\cdots g_n)=\Delta(L[\{1,2,..,n\}])=\sum_{I+J=\{1,2,..,n\}} L[I]\ot L[J] \ .
\end{equation}    
From (\mref{subwords_of_primitive}) one gets also that $j_\B$ is a morphism of bialgebras. If for any reason, there exists a lifting of 
\begin{equation}
\Delta_\B\circ i_{\B,\A} 	
\end{equation}
as a comultiplication of $\A$, then $j_\B$ is into (see the statement and the proof below). Formula (\mref{subwords_of_primitive}) proves that we have the following maps (save the -- hypothetical -- dotted one). 
\begin{figure}[h]
\centering
\begin{tikzpicture}
  \matrix (m) [matrix of math nodes,row sep=3em,column sep=8em,minimum width=2em]
  {\ncp{A}{G} & \A\\
\ncp{A}{G}\otimes \ncp{A}{G} & \A\otimes \A &\\
};
\path[-stealth] 
    (m-1-1) edge [] node [above] {$s_G$} (m-1-2)
    (m-1-1) edge [] node [left] {$\Delta_{\shuffle}$} (m-2-1)
    (m-1-2) edge [dashed,->] node [right] {$\Delta_\A$} (m-2-2)
    (m-2-1) edge [] node [above] {$s_G\otimes s_G$} (m-2-2); 
\end{tikzpicture}
\caption{The unique lifting $\Delta_\A$ (when it exists).}\label{hypothetic}
\end{figure}

Where $G\in\Prim(\B)$ is any generating set of the AAU $\A$. 
We emphasize the fact that, in the diagram above, $G$ must be understood set-theoretically (i.e. with no relation between the elements\footnote{We will see, below and in paragraph \ref{cex} how it is crucial to consider that $[\lambda x]$ and $\lambda [x]$ are not necessarily equal, when $\lambda x\in G$ (for clarity, $[y]\in\ncp{A}{G}$ is the image of $y\in G$).}).\\ 
In fact, one has the following proposition
\begin{proposition}\label{injectivity} Let $\B$ be a bialgebra over a (commutative) $\Q$-algebra $A$, the notations being those of figures \ref{primitive_alg} and \ref{hypothetic}, then the following statements are equivalent\\
i) For a generating set $G\subset Prim(\B)$, 
$ker(s_G)\subset ker(s_G\otimes s_G)\circ \Delta_{\shuffle}$.\\
ii) For any generating set $G\subset Prim(\B)$, $ker(s_G)\subset ker(s_G\otimes s_G)\circ \Delta_{\shuffle}$.\\
iii) $j_\B$ is into.
\end{proposition}
\begin{proof}
$i)\Longrightarrow iii)$ In order to prove this, we need to construct the arrows $\sigma,\tau$ which are a decomposition of a section of $j_\B$. 
\begin{figure}[h!]
\centering
\begin{tikzpicture}
  \matrix (m) [matrix of math nodes,row sep=3em,column sep=8em,minimum width=2em]
  {\Prim(\B) & \A & \B\\
\U(\Prim(\B))& \Tens(\Prim(\B))&\\
};
\path[right hook->] 
    (m-1-1) edge [] node [above] {$i_{\A,P}$} (m-1-2)
            edge [] node [left] {$i_{\U,P}$} (m-2-1)
    (m-1-2) edge [] node [above] {$i_{\B,\A}$} (m-1-3) 
            edge [dashed,->] node [right] {$\sigma$} (m-2-2) ;
  \path[-stealth]             
    (m-2-1) edge node [above] {$j_{\B}$}  (m-1-2)
    (m-2-2) edge [dashed,->] node [above] {$\tau$} (m-2-1) ;
\end{tikzpicture}
\caption{The sub-bialgebra $\A$ generated by primitive elements.}\label{primitive elements}
\end{figure}
%\end{proof}
%     
Let us remark that, when $\Prim(\B)$ is free as an $A$-module, the proof of this fact is a consequence of the PBW theorem\footnote{See \cite{B_Lie_II_III} Ch2 \S 1 n\up{o} 6 th 1  for a field of characteristic zero and \S 1 Ex. 10 for the free case (over a ring $A$ with 
$\Q\subset A$).}. But, here, we will construct the section in the general case using projectors which are now classical for the free case but which still can be computed analytically \cite{reutenauer} as they lie in $\Q[[X]]$ and still converge in $\A$. 

(Injectivity of $j_\B$, construction of the section $\tau\circ\sigma$)\pointir\\
As $\A$ is the subalgebra of $\B$ generated by $\Prim(\B)$, one has $\Im(j_\B)=\A$.\\ 
Remark that all series $\sum_{n\geq 0}a_n (I_+)^{*n}$ are summable on $\A$ (not in general on $\B$ for example in case $\B$ contains non-trivial group-like elements).\\
We define 
\begin{equation}
	c=\log_*(I)=\sum_{n\geq 1}\frac{(-1)^{n-1}}{n} (I_+)^{*n}
\end{equation}
and remark that, in view of Prop. (\ref{diagonal_rep}), in the case when $\B=\ncp{A}{X}$ one has $\A=\B$ and, with $S(X)=\log(1+X)$  
\begin{eqnarray}\label{diagonal_pi1}
&&\sum_{w\in X^*} w\ot \pi_{1,\A}(w)=\rho(\log(I))=\rho(S(I^+))=S(\rho(I^+))=\cr
&&S(\sum_{w\in X^*\atop w\not= 1_{X^*}} w\ot w)=S(\calD_X-1_{X^*}\ot 1_{X^*})=\log(\calD_X)\ .
\end{eqnarray}

We first prove that $\pi_{1,\A}$ is a projector $\A\ra \Prim(\B)$. The key point is that $\Delta_\A$ (the restriction of the comultiplication to $\A$) is a morphism of bialgebras \footnote{In fact it is the case for any cocommutative bialgebra, be it generated by its primitive elements or not.} $\A\ra \A\ot \A$. 
We first prove that $\Delta_\A$ ``commutes'' with the convolution. This is a  consequence of the following property
\begin{lemma}\label{lemma1}
i) Let $f_i\in \End(\B_i)$, be such that 
$\varphi f_1 = f_2 \varphi$.
\begin{figure}[h!]
\centering
\begin{tikzpicture}
  \matrix (m) [matrix of math nodes,row sep=3em,column sep=8em,minimum width=2em]
  {\B_1 & \B_2\\
   \B_1&   \B_2\\
};
  \path[-stealth]             
    (m-1-1) edge node [above] {$\varphi$}  (m-1-2)
            edge node [left] {$f_1$}  (m-2-1)
    (m-2-1) edge node [below] {$\varphi$}  (m-2-2)
    (m-1-2) edge node [right] {$f_2$}  (m-2-2);
\end{tikzpicture}
\caption{Intertwining with a morphism of bialgebras (the functions of $f_i$ below will be computed with the respective convolution products).}\label{intertwining}
\end{figure}

i) Then, if $P\in A[X]$, one has 
\begin{equation}\label{conv_pol}
\varphi P(f_1)=P(f_2)\varphi\ .
\end{equation}
ii) If the series $\sum_{n\geq 0} (I_{(i)}^+)^{*n},\ i=1, 2$ are summable and, 
if $f_1(1)=0$ (which implies $f_2(1)=0$) and $S\in A[[X]]$, the families 
$(\scal{S}{X^n}f_i^{*n})_{n\in\N}$ are summable, we denote by  
$S(f_i)$ their sums (note that this definition is coherent with 
the previous ones when $S$ is a polynomial).\\
One has, for the convolution product, 
\begin{equation}\label{conv_ser}
\varphi S(f_1)=S(f_2)\varphi\ .
\end{equation}
\end{lemma}
\begin{proof} The only delicate part is (ii). First, one remarks that, if $\varphi$ is a morphism of bialgebras, one has 
\begin{equation}
(\varphi\ot \varphi)\circ \Delta_1^+=\Delta_2^+\circ \varphi 
\end{equation}
then, the image by $\varphi$ of an element of order less than $N$ (i.e. such that 
$\Delta_1^{+(N)}(x)=0$) is of order less than $N$. Let now $S$ be a univariate 
series $S=\sum_{k=0}^\infty a_k X^k$. For every element $x$ of order less than $N$ and 
$f\in End(\B)$, one has 
\begin{eqnarray}
S(f)(x)&=&\sum_{k=0}^\infty a_k f^{*k}(x)=\sum_{k=0}^\infty a_k \mu^{(k-1)} f^{\ot k}\Delta^{(k-1)}(x)\cr
&=&\sum_{k=0}^\infty a_k \mu^{(k-1)} (f^{\ot k})\circ (I_+^{\ot k})\Delta^{(k-1)}(x)\cr
&=&\sum_{k=0}^N a_k \mu^{(k-1)} (f^{\ot k})\Delta_+^{(k-1)}(x)\ .
\end{eqnarray}
This proves, in view of (i) that $\varphi\circ S(f_1)=S(f_2)\circ \varphi$. 

Thanks to Lemma \ref{lemma1}, we can now prove that $\pi_1$ is a projector $\B\ra \Prim(\B)$.\\
In case $\B$ is cocommutative, the comultiplication $\Delta$ is a morphism of bialgebras, so one has  
\begin{equation}\label{prim1}
\Delta\circ log_*(I) = log_*(I\ot I) \circ \Delta\ .
\end{equation}
But
\begin{eqnarray}
log_*(I\ot I) &= log_*((I\ot e) * (e \ot I))\cr
& = log_*(I\ot e) + log_*(e\ot I)\cr
& = log_*(I)\ot  e + e \ot log_*(I)\ .
\end{eqnarray}

Then 
\begin{equation}
\Delta(log_*(I)) =  \left( log_*(I)\ot  e + e \ot log_*(I) \right) \circ \Delta
\end{equation}
which implies that $log_*(I)(\B) \subset Prim(\B)$. To finish the proof that $\pi_1$ is a projector onto $Prim(\B)$, it suffices to remark that, for $x\in Prim(\B)$ and $n\geq 2$,  $(\mathrm{Id}^+)^{*n}(x)=0$ then 
\begin{equation}
	log_*(I)(x)=\mathrm{Id}^+(x)=x\ .
\end{equation}
\end{proof}
Now, we consider
\begin{equation}
	I_\A=exp_*(log_*(I_\A))=\sum_{n\geq 0}\frac{1}{n!} \pi_{1,\A}^{*n}\ ,
\end{equation}
where $\pi_{1,[\A]}=log_*(I_\A)$.\\ 
Let us prove that the summands form a resolution of unity.\\ 
First, one defines $\A_{[n]}$ as the linear span of the powers $\{P^n\}_{P\in \Prim(\B)}$ or, equivalently, of the symmetrized products
\begin{equation}
	\frac{1}{n!} \sum_{\sigma\in \SG_n} P_{\sigma(1)}P_{\sigma(2)}\cdots P_{\sigma(n)}\ .
\end{equation}
It is obvious that $\Im(\pi_{1,\A})^{*n})\subset \A_{[n]}$. We remark that 
\begin{equation}
	\pi_{1,\A}^{*n}=\mu_\B^{(n-1)}\pi_{1,\A}^{\ot n}\Delta^{(n-1)}=
	\mu_\B^{(n-1)}\pi_{1,\A}^{\ot n}I_+^{\ot n}\Delta^{(n-1)}=
	\mu_\B^{(n-1)}\pi_{1,\A}^{\ot n}\Delta_+^{(n-1)}
\end{equation}
as $\pi_{1,\A}I_+=\pi_{1,\A}$.
Now, let $P\in \Prim(\A)$. We compute $\pi_{1,\A}^{*n}(P^m)$. Indeed, if $m<n$, one has 
\begin{equation}
	\pi_{1,\A}^{*n}(P^m)=\mu_\B^{n-1}\Delta_+^{n-1}(P^m)=0\ .
\end{equation}
If $n=m$, one has, from (\ref{subwords_of_primitive})
\begin{equation}
	\Delta_+^{n-1}(P^n)=n! P^{\ot n}
\end{equation}
and hence $\pi_{1,\A}^{*n}$ is the identity on $\A_{[n]}$. If $m>n$, the nullity of $\pi_{1,\A}^{*n}(P^m)$ is a consequence of the following lemma. 
\begin{lemma}\label{key_for_ortho}
 Let $\B$ be a bialgebra and $P$ a primitive element of $\B$. Then\\ 
i) The series $log_*(I)$ is summable on each power $P^m$\\
ii) $log_*(I)(P^m)=0$ for $m>2$ 
\end{lemma}
\begin{proof} i) As $\Delta_+^{*N}(P^m)=0$ for $N>m$, one has $I_+^{*N}(P^m)=0$ for these values.\\ 
ii) Let $a$ be a letter, the morphism of AAU $\varphi_P : A[a]\ra \B$, defined by 
\begin{equation}
	\varphi_P(a)=P
\end{equation}
is, in fact, a morphism of bialgebras.
\begin{figure}[h!]
\centering
\begin{tikzpicture}
  \matrix (m) [matrix of math nodes,row sep=3em,column sep=8em,minimum width=2em]
  {A[a] & \B\\
   A[a]&   \B\\
};
  \path[-stealth]             
    (m-1-1) edge node [above] {$\varphi_P$}  (m-1-2)
            edge node [left] {$I_{A[a]}^+$}  (m-2-1)
    (m-2-1) edge node [below] {$\varphi_P$}  (m-2-2)
    (m-1-2) edge node [right] {$I_\B^+$}  (m-2-2);
\end{tikzpicture}
\caption{Intertwining with one primitive element.}
\end{figure}
One checks easily that $\pi_{1,[A[a]]}(a^m)=0$ for $m>2$ which is a consequence of the general equality (see \mref{diagonal_pi1})
\begin{equation}
	\sum_{w\in X^*}(w\ot \pi_1(w))=log(\sum_{w\in X^*}w\ot w)
\end{equation}
because, for $Y=\{a\}$ (and then $\ncp{A}{X}=A[a]$) one has\newpage    
\begin{eqnarray}
\log(\sum_{w\in X^*}w\ot w)=\log(\sum_{n\geq 0} a^n\ot a^n)&=&\cr
\log(\sum_{n\geq 0}\frac{1}{n!}(a\ot a)^{(\shuffle\ot conc)\, n})=
\log(\exp(a\ot a))&=& a\ot a
\end{eqnarray}
this proves that $\pi_{1,\A}^{*n}(\A_{[m]})=0$ for $m\not= n$ and hence the summands of the sum 
\begin{equation}
		I_\A=exp_*(log_*(I_\A))=\sum_{n\geq 0}\frac{1}{n!} \pi_{1,\A}^{*n}\ .
\end{equation}
are pairwise orthogonal projectors with $\Im(\pi_{1,\A}^{*n})=\A_{[n]}$ and then 
\begin{equation}
	\A=\oplus_{n\geq 0}\, \A_{[n]}\ .
\end{equation}
This decomposition enables to construct $\sigma$ by 
\begin{equation}
	\sigma(P^n)=\frac{1}{n!}\Delta_+^{(n-1)}(P^n)\in T_n(\Prim(\B))
\end{equation}
for $n\geq 1$ and, one sets $\sigma(1_\B)=1_{T(\Prim(\B)}$.\\ 
It is easy to check that $j_\B\circ\tau\circ\sigma=Id_\A$ as $\A$ is (linearly) generated by the powers 
$(P^m)_{P\in \Prim(\B),m\geq 0}$. 
\end{proof}
{\bf End of the proof of proposition \ref{injectivity}}\pointir\\
$iii)\Longrightarrow ii)$ If ${j_\B}$ is into, then $i_{\U,\A}$ is one-to-one and one gets a comultiplication 
$$
\Delta_\A\ :\ \A\ra \A\otimes \A
$$ 
such that, for any list of primitive elements $L=[g_1,g_2,\cdots g_n]$ (the denotations are the same as previously)
\begin{equation}
\Delta_\A(g_1g_2\cdots g_n)=\Delta(L[\{1,2,..,n\}])=\sum_{I+J=\{1,2,..,n\}} L[I]\ot_\A L[J] 
\end{equation}    
 but, this time, the tensor product $\ot_\A$ is understood as being in $\A\ot \A $. This guarantees that the diagram Fig. \ref{hypothetic} commutes for any $G$.\\
$ii)\Longrightarrow i)$ Obvious.
\end{proof}
\subsection{Counterexamples and discussion}\label{cex}

\subsubsection{Counterexamples}

It has been said that, with $\B=(\Q[\ep][x],.,1_{\Q[\ep][x]},\Delta,c)$ (notations as above), $j_\B$ is not into, let us show this statement.\\ 
The $q$-infiltration coproduct \cite{luque} $\Delta_q$ is defined on the free algebra $\ncp{K}{X}$ ($K$ is a unitary ring), by its values on the letters 
\begin{equation}
\Delta_q(x)=x\ot 1+1\ot x+q(x\ot x)	
\end{equation}
where $q\in K$. One can show easily that, for a word $w\in X^*$,
\begin{equation}\label{combinf}
\Delta_q(w)=\sum_{I\cup J=[1..|w|]}q^{|I\cap J|} w[I]\ot w[J]	
\end{equation}
with, as above (for $I=\{i_1<i_2<..<i_k\}\subset \{1,2,..,n\}$ and $w=a_1a_2\cdots a_n$), 
$w[I]=a_{i_1}a_{i_2}\cdots a_{i_k}$.\\ 
Then, with $K=\Q[\ep],\ q=\ep,\ X={x}$, one has (as a direct application of Eq. \ref{combinf})  
\begin{equation}
\Delta_\ep(x^n)=\sum_{k=0}^n \binomial{n}{k} x^k\ot x^{n-k} + \ep\, \sum_{k=1}^n k\binomial{n}{k} x^k\ot x^{n-k+1}\ . 	
\end{equation}
This proves that, here, the space of primitive elements is a submodule of $K.x$ and solving $\Delta_\ep(\lambda x)=(\lambda x)\ot 1+ 1\ot (\lambda x)$, one finds $\lambda=\lambda_1\ep$. Together with $\ep\, x\in Prim(\B)$ this proves that $Prim(\B)$ is of $\Q$-dimension one (in fact equal to $\Q.(\ep\,x)$). Now, the consideration of the morphism of Lie algebras $Prim(\B)\ra K[x]/(\ep K[x])$ which sends $\ep\,x$ to $x$ proves that, in $\U(Prim(\B))$, we have $(\ep\,x)(\ep\,x)\not=0$ and $j_\B$ cannot be into.\\
For a graded counterexample\footnote{This example is due to Darij Grinberg.}, one can see that, with\\ 
$K=\Q[\ep],\ X=\{x,y,z\},\ \B=\ncp{K}{X}$ and 
\begin{equation}
\Delta(x)=x\ot 1+1\ot x+\ep\,(y\ot z),\ \Delta(y)=y\ot 1+1\ot y,\ 	
\Delta(z)=z\ot 1+1\ot z
\end{equation}   
the same phenomenon occurs (for the gradation, one takes\\ $\deg(y)=\deg(z)=1,\ \deg(x)=2$). 
\subsubsection{The theorem from the point of view of summability}
{\bf From now on, the morphism $j_\B$ is supposed into}.\\
The bialgebra $\B$ being supposed cocommutative, we discuss the equivalent conditions under which we are in the presence of an enveloping algebra i.e. 
\begin{equation}
{\cal B}\cong_{A-bialg}\mathcal{U}(Prim({\cal B}))
\end{equation}
from the point of view of the convergence of the series $log_*(I)$\footnote{In a $A$-bialgebra, one can always consider the series of endomorphisms 
\begin{equation}
	\sum_{n\geq 1} \frac{(-1)^{n-1}}{n} (I^+)^{*n}\ .
\end{equation}
The family $(\frac{(-1)^{n-1}}{n} (I^+)^{*n})_{n\geq 0}$ is summable 
iff $((I^+)^{*n})_{n\geq 0}$ is (use \mref{scaling}). 
}. 
These conditions are known as the theorem of Cartier-Quillen-Milnor-Moore (CQMM). 
\begin{theorem}\cite{B_Lie_II_III}
Let ${\cal B}$ be a $A$-cocommutative bialgebra ($A$ is a $\Q$-AAU) and $\A$, as above, the subalgebra generated by $\Prim(\B)$. 
Then, the following conditions are equivalent :
\begin{enumerate}
\item ${\cal B}$ admits an increasing filtration 
\begin{eqnarray*}
	{\cal B}_0=A.1_{\cal B}\subset{\cal B}_1\subset\cdots\subset{\cal B}_n\subset{\cal B}_{n+1}\cdots
\end{eqnarray*}
compatible with the structures of algebra (i.e. for all $p,q\in\N$, one has ${\cal B}_p{\cal B}_q\subset{\cal B}_{p+q}$) and coalgebra~:
\begin{eqnarray*}
\forall n\in\N,&&\Delta({\cal B}_n)\subset \sum_{p+q=n}{\cal B}_p\otimes{\cal B}_q.
\end{eqnarray*}
\item $((\mathrm{Id}^+)^{*n})_{n\in \N}$ is summable in $\End(\B)$.
\item $\B=\A$.
\end{enumerate}
\end{theorem}

\begin{proof} We prove 
\begin{equation}
\mathrm{(ii)}\Longrightarrow\mathrm{(iii)}\Longrightarrow\mathrm{(i)}\Longrightarrow\mathrm{(ii)}
\end{equation}
$\mathrm{(ii)}\Longrightarrow\mathrm{(iii)}$\pointir\\

The image of $j_\B$ it is the subalgebra generated by the primitive elements. Let us prove that, when   
$((\mathrm{Id}^+)^{*n})_{n\in \N}$ is summable, one has $\Im(j_\B)=\B$. The series $log(1+X)$ is without constant term so, in virtue of (\ref{subs}) and the summability of  $((\mathrm{Id}^+)^{*n})_{n\in \N}$, one has 
\begin{equation}
	exp(log(e+\mathrm{Id}^+))=exp(log(1+X))({Id}^+)=1_{\End(\B)}+{Id}^+=e+{Id}^+=I
\end{equation}
Set $\pi_1=log(e+\mathrm{Id}^+)$. 

To end this part, let us compute, for $x\in \B$
\begin{equation}
x=exp(\pi_1)(x)=(\sum_{n\geq 0} \frac{1}{n!} \pi_1^{*n})(x)=
(\sum_{n=0}^N \frac{1}{n!} \mu^{(n-1)}\pi_1^{\ot n})\Delta^{(n-1)}(x)
\end{equation}
where $N$ is the first order for which $\Delta^{+(n-1)}(x)=0$ (as $\pi_1\circ\mathrm{Id}^+=\pi_1$). This proves that $\B$ is generated by its 
primitive elements.\\ 
The implications $\mathrm{(iii)}\Longrightarrow\mathrm{(i)}$ and $\mathrm{(i)}\Longrightarrow\mathrm{(ii)}$ are obvious.\\
\end{proof}

\begin{remark} i) The equivalence $(i)\Longleftrightarrow (iii)$ is the classical CQMM theorem (see \cite{B_Lie_II_III}). The equivalence with (ii) could be called the ``Convolutional CQMM theorem''. The combinatorial aspects of this last one will be the subject of a forthcoming paper.\\
ii) When $\Prim(\B)$ is free, we have ${\cal B}\cong_{k-bialg}\mathcal{U}(Prim({\cal B}))$ and $\B$ is an enveloping algebra.\\
iii) The (counter) example is the following with $A=k[x]$ ($k$ is a field of characteristic zero). Let $Y$ be an alphabet and $\ncp{A}{Y}$ be the usual free algebra (the space of non-commutative polynomials over $Y$) and $\epsilon$, the ``constant term'' linear form. Let $conc$ be the concatenation and $\Delta$ the dual law of the shuffle product (cf supra). 

Then the bialgebra $(\ncp{A}{Y},conc,1_{Y^*},\Delta,\epsilon)$ is a Hopf algebra (it is the enveloping algebra of the Lie polynomials). Let 
$\ncp{A_+}{Y}=ker(\epsilon)$ and, for $N\geq 2$ $J_N=x^N.\ncp{A_+}{Y}$ then, $J_N$ is a Hopf ideal and $Prim(\ncp{A}{Y}/(J_N))$ is never free (no basis). 
\end{remark}

\section{Application to the $\phi$-deformed shuffle.}
\subsection{General results for the $\phi$-deformed shuffle.}

Let $Y=\{y_i\}_{i\in I}$ be still a totally ordered alphabet and $\AY$ be equipped 
with the $\phi$-deformed stuffle defined by \cite{orsay}
\begin{itemize}
\item[i)] for any $w\in Y^*$, $1_{Y^*}{\stuffle_\phi}w=w{\stuffle_\phi}1_{Y^*}=w$, 
\item[ii)] for any $y_i,y_j\in Y$ and $u,v\in Y^*$, 
\begin{eqnarray}\label{recursion}
y_iu{\stuffle_\phi}y_jv&=&y_j(y_iu{\stuffle_\phi}v)+y_i(u{\stuffle_\phi} y_jv)+\phi(y_i,y_j)u{\stuffle_\phi}v,
\end{eqnarray}
where $\phi$ is an arbitrary mapping
\begin{eqnarray*}
\phi: Y\times Y&\longrightarrow&AY\ .
\end{eqnarray*}
\end{itemize}
\begin{definition}
Let
\begin{eqnarray*}
\phi: Y\times Y&\longrightarrow&AY
\end{eqnarray*}
be defined by its structure constants
\begin{eqnarray*}\label{phi_struct_const}
(y_i,y_j)&\longmapsto&\phi(y_i,y_j)=\sum_{k\in I}\gamma_{i,j}^k\,y_k.
\end{eqnarray*}
\end{definition}
\begin{proposition}
The recursion (\ref{recursion}) defines a unique mapping
\begin{eqnarray*}
\stuffle_\phi\,: Y^*\times Y^*&\longrightarrow&\ncp{A}{Y}.
\end{eqnarray*}
\end{proposition}
\begin{proof}
Let us denote $(Y^*\times Y^*)_{\leq n}$ the set of words $(u,v)\in Y^*\times Y^*$ such that $|u|+|v|\leq n$. We construct a sequence of mappings 
$$
{\stuffle_\phi}_{\leq n} : (Y^*\times Y^*)_{\leq n} \longrightarrow \ncp{A}{Y}.
$$
which satisfy the recursion of \mref{recursion}. For $n=0$, we have only a pre-image and 
${\stuffle_\phi}_{\leq 0}(1_{Y^*})=1_{Y^*}\otimes 1_{Y^*}$. Suppose ${\stuffle_\phi}_{\leq n}$ already constructed and let\\ 
$(u,v)\in (Y^*\times Y^*)_{\leq n+1}\setminus (Y^*\times Y^*)_{\leq n}$, i.e. $|u|+|v|=n+1$.\\ 
One has three cases : $u=1_{Y^*},\ v=1_{Y^*}$ and $\ (u,v)\in Y^+\times Y^+$. For the first two, one uses the initialisation of the recursion, thus 
$$
{\stuffle_\phi}_{\leq n+1}(w,1_{Y^*})={\stuffle_\phi}_{\leq n+1}(1_{Y^*},w)=w\ .
$$
For the last case, write $u=y_iu',\ v=y_jv'$ and use, to get 
$$
{\stuffle_\phi}_{\leq n+1}(y_iu',y_jv')=y_i{\stuffle_\phi}_{\leq n}(u',y_jv')+ y_j{\stuffle_\phi}_{\leq n}(y_iu',v')+y_{i+j}{\stuffle_\phi}_{\leq n}(u',v')
$$
this proves the existence of the sequence $({\stuffle_\phi}_{\leq n})_{n\geq 0}$. Every ${\stuffle_\phi}_{\leq n+1}$ extends the preceding so there is a mapping  
$$
{\stuffle_\phi} : Y^*\times Y^*\longrightarrow \ncp{A}{Y}.
$$
which extends all the ${\stuffle_\phi}_{\leq n+1}$ (the graph of which is the union of the graphs of the ${\stuffle_\phi}_{\leq n}$). This proves the existence.
 For unicity, just remark that, if there were two mappings ${\stuffle_\phi},{\stuffle'_\phi}$, the fact that they must fulfil the recursion (\ref{recursion}) implies that ${\stuffle_\phi}={\stuffle'_\phi}$.
\end{proof}
We still denote by $\phi$ and ${\stuffle_\phi}$ the linear extension of $\phi$ and ${\stuffle_\phi}$ to $AY\otimes AY$  
and $\AY\otimes\AY$ respectively.\\
Then  ${\stuffle_\phi}$ is a law of algebra (with $1_{Y^*}$ as unit) on $\AY$.
\begin{lemma}
 Let $\Delta$ be the morphism $\ncp{A}{Y}\ra \ncs{A}{Y^*\otimes Y^*}$ defined on the letters by 
\begin{equation}
	\Delta(y_s)=y_s\otimes 1+1\otimes y_s+\sum_{n,m\in I} \gamma_{n,m}^s\, y_n\otimes y_m\ .
\end{equation}
Then
\begin{itemize}
\item[i)] for all $w\in Y^+$ we have 
\begin{equation}
	\Delta(w)=w\otimes 1+1\otimes w+\sum_{u,v\in Y^+} \scal{\Delta(w)}{u\otimes v}\, u\otimes v  
\end{equation}
\item[ii)]for all $u,v,w\in Y^*$, one has
\begin{equation}
	\scal{u\stuffle_\phi v}{w}=\scal{u\otimes v}{\Delta(w)}^{\otimes\, 2}\ .
\end{equation}
\end{itemize}
\end{lemma}
\begin{proof}
\begin{itemize}
\item[i)] By recurrence on $|w|$. If $w=y_s$ is of length one, it is obvious from the definition. If $w=y_sw'$, we have, from the fact that $\Delta$ is a morphism 
\begin{eqnarray}
	\Delta(w)&=&\biggl(y_s\otimes 1+1\otimes w+\sum_{i,j\in I}\gamma_{i,j}^s y_i\otimes y_j\biggr)\cr
	&&\biggl(w'\otimes 1+1\otimes w'+\sum_{u,v\in Y^+}\scal{u\otimes v}{\Delta(w')}\biggr) 
\end{eqnarray}
the development of which proves that $\Delta(w)$ is of the desired form.\\ 
\item[ii)] Let $S(u,v):=\sum_{w\in Y^*}\scal{u\otimes v}{\Delta(w)}\, w$. It is easy to check (and left to the reader) that, for all $u\in Y^*$, $S(u,1)=S(1,u)=u$. Let us now prove that, for all $y_i,y_j\in Y$ and $u,v\in Y^*$
\begin{equation}
S(y_iu,y_jv)=y_iS(u,y_jv)+y_jS(y_iu,v)+\phi(y_i,y_j)S(u,v)\ .	
\end{equation}
Indeed, noticing that $\Delta(1)=1\otimes 1$, one has 
\begin{align*}
S(y_iu,y_jv)&=\sum_{w\in Y^*}\scal{y_iu\otimes y_jv}{\Delta(w)}w=
 						\sum_{w\in Y^+}\scal{y_iu\otimes y_jv}{\Delta(w)}w\cr
						&=
						\sum_{y_s\in Y,\ w'\in Y^*}\scal{y_iu\otimes y_jv}{\Delta(y_sw')}\, y_sw'\cr
%						&=\sum_{y_s\in Y,\ w'\in Y^*}\scal{y_iu\otimes y_jv}
%                                                               {\biggl(y_s\otimes 1+1\otimes y_s\cr
%&\qquad\qquad\qquad\qquad+\sum_{n,m\in I}\gamma_{n,m}^s\,y_n\otimes y_m\biggr)\Delta(w')}\, y_sw'\cr
						&=\sum_{y_s\in Y,\ w'\in Y^*}\scal{y_iu\otimes y_jv}
                                                               {\biggl(y_s\otimes 1+1\otimes y_s+\sum_{n,m\in I}\gamma_{n,m}^s\,y_n\otimes y_m\biggr)\Delta(w')}\, y_sw'\cr
						&= \sum_{y_s\in Y,\ w'\in Y^*}\scal{y_iu\otimes y_jv}
{(y_s\otimes 1)\Delta(w')}\, y_sw'\cr
						&+
\sum_{y_s\in Y,\ w'\in Y^*}\scal{y_iu\otimes y_jv}
{(1\otimes y_s)\Delta(w')}\, y_sw'\cr
						&+\sum_{y_s\in Y,\ w'\in Y^*}\scal{y_iu\otimes y_jv}
{(\sum_{n,m\in I}\gamma_{n,m}^s\,y_n\otimes y_m)\Delta(w')}\, y_sw'
\end{align*}
\begin{align*}
						&= \sum_{w'\in Y^*}\scal{u\otimes y_jv}{\Delta(w')}\, y_iw'+
						\sum_{w'\in Y^*}\scal{y_iu\otimes v}{\Delta(w')}\, y_jw'\cr
						&+ \sum_{y_s\in Y,\ w'\in Y^*}\scal{u\otimes v}
{\gamma_{i,j}^s\Delta(w')}\, y_sw'\cr
						&=y_i\sum_{w'\in Y^*}\scal{u\otimes y_jv}{\Delta(w')}\, w'+
						y_j\sum_{w'\in Y^*}\scal{y_iu\otimes v}{\Delta(w')} w'\cr
						&+ \sum_{y_s\in Y}\gamma_{i,j}^s\, y_s\sum_{w'\in Y^*}\scal{u\otimes v}
{\Delta(w')}\,w'\cr
						&=y_iS(u,y_jv)+	y_jS(y_iu,v)+ \phi(y_i,y_j)S(u,v)
\end{align*} 
then the computation of $S$ shows that, for all $u,v\in Y^*$, $S(u,v)=u\stuffle_\phi v$ as $S$ is bilinear, so $S=\stuffle_\phi$. 
\end{itemize}
\end{proof}
\begin{theorem}
\begin{itemize}
\item[i)] The law ${\stuffle_\phi}$ is commutative if and only if the extension
\begin{eqnarray*}
\phi : AY\otimes AY\longrightarrow{A}{Y}
\end{eqnarray*}
is so.
\item[ii)] The law ${\stuffle_\phi}$ is associative if and only if the extension
\begin{eqnarray*}
\phi : AY\otimes AY\longrightarrow{A}{Y}
\end{eqnarray*}
is so.
\item[iii)] Let $\gamma_{x,y}^z:=\langle{\phi(x,y)}|{z}\rangle$ be the structure constants of $\phi$ (w.r.t. the basis $Y$),
then ${\stuffle_\phi}$ is dualizable if and only if $(\gamma_{x,y}^z)_{x,y,z\in X}$ has the following decomposition property\footnote{One can prove that, in case $Y$ is a semigroup, the associated $\phi$ fulfils \mref{cond_D} iff $Y$ fulfils ``condition D'' of Bourbaki (see \cite{B_alg_I_III})}
\begin{eqnarray}\label{cond_D}
	(\forall z\in X)(\#\{(x,y)\in X^2|\gamma_{x,y}^z\not=0\}<+\infty)\ .
\end{eqnarray}
\end{itemize}
\end{theorem}
\begin{proof}
(i) First, let us suppose that $\phi$ be commutative and consider $T$, the twist, i.e. the operator in $\ncs{A}{Y^*\otimes Y^*}$ defined by 
\begin{equation}
	\scal{T(S)}{u\otimes v}=\scal{S}{v\otimes u}\ .
\end{equation}
It is an easy check to prove that $T$ is a morphism of algebras. If $\phi$ is commutative, then so is the following diagram. 
\begin{center}
\begin{tikzpicture}
  \matrix (m) [matrix of math nodes,row sep=3em,column sep=8em,minimum width=2em]
  {Y & \ncs{A}{Y^*\otimes Y^*} \\
     %\ncs{A}{Y^*\otimes Y^*} 
     & \ncs{A}{Y^*\otimes Y^*} \\
     };
  \path[-stealth]
    (m-1-1) edge [] node [above] {$\Delta_{\ministuffle_\phi}$} (m-1-2)
            edge [] node [below] {$\Delta_{\ministuffle_\phi}$} (m-2-2)
    (m-1-2) edge node [right] {$T$}  (m-2-2);
\end{tikzpicture}
\end{center}
and, then, the two morphisms $\Delta_{\ministuffle_\phi}$ and $T\circ\Delta_{\ministuffle_\phi}$ coincide on the generators $Y$ of the algebra $\ncp{A}{Y}$ and hence over $\ncp{A}{Y}$ itself. Now for all $u,v,w\in Y^*$, one has
\begin{eqnarray}
&&	\scal{v\ministuffle_\phi u}{w}=\scal{v\otimes u}{\Delta_{\ministuffle_\phi}(w)} =
	\scal{u\otimes v}{T\circ \Delta_{\ministuffle_\phi}(w)}\cr
&=&\scal{u\otimes v}{\Delta_{\ministuffle_\phi}(w)}=\scal{u\ministuffle_\phi v }{w}
\end{eqnarray}
which proves that $v\ministuffle_\phi u=u\ministuffle_\phi v$. Conversely, if $\ministuffle_\phi$ is commutative, one has, for $i,j\in I$ 
\begin{equation}
	\phi(y_j,y_i)=y_j\ministuffle_\phi y_i - (y_j\shuffle y_i)=y_i\ministuffle_\phi y_j - (y_i\shuffle y_j)=\phi(y_i,y_j)\ .
\end{equation}
(ii) Likewise, if $\phi$ is associative, let us define the operators 
\begin{equation}
\overline{\Delta_{\ministuffle_\phi}\otimes I} : 
\ncs{A}{Y^*\otimes Y^*} \ra \ncs{A}{Y^*\otimes Y^*\otimes Y^*}
\end{equation} 
by 
\begin{equation}
\scal{\overline{\Delta_{\ministuffle_\phi}\otimes I}(S)}{u\otimes v\otimes w}=
\scal{S}{(u\ministuffle_\phi v)\otimes w}
\end{equation} 
and, similarly, 
\begin{equation}
\overline{I\otimes \Delta_{\ministuffle_\phi}} : 
\ncs{A}{Y^*\otimes Y^*} \ra \ncs{A}{Y^*\otimes Y^*\otimes Y^*}
\end{equation} 
by 
\begin{equation}
\scal{\overline{I\otimes \Delta_{\ministuffle_\phi}}(S)}{u\otimes v\otimes w}=
\scal{S}{u\otimes (v\ministuffle_\phi  w)}
\end{equation} 
it is easy to check by direct calculation that they are well defined morphisms and that the following diagram   
\begin{center}
\begin{tikzpicture}
  \matrix (m) [matrix of math nodes,row sep=3em,column sep=8em,minimum width=2em]
  {
     Y & \ncs{A}{Y^*\otimes Y^*} \\
     \ncs{A}{Y^*\otimes Y^*} & \ncs{A}{Y^*\otimes Y^*\otimes Y^*} \\};
  \path[-stealth]
    (m-1-1) edge node [left] {$\Delta_{\ministuffle_\phi}$} (m-2-1)
            edge [] node [above] {$\Delta_{\ministuffle_\phi}$} (m-1-2)
    (m-2-1.east|-m-2-2) edge node [below] {$\overline{\Delta_{\ministuffle_\phi}\otimes I}$}
            node [above] {} (m-2-2)
    (m-1-2) edge node [right] {$\overline{I\otimes \Delta_{\ministuffle_\phi}}$} (m-2-2);
\end{tikzpicture}
\end{center}
is commutative. This proves that the two composite morphisms 
$$
\overline{\Delta_{\ministuffle_\phi}\otimes I}\circ \Delta_{\ministuffle_\phi}
$$ 
and 
$$
\overline{I\otimes \Delta_{\ministuffle_\phi}}\circ \Delta_{\ministuffle_\phi}
$$
coincide on $Y$ and then on $\ncp{A}{Y}$. Now, for $u,v,w,t\in Y^*$, one has 
\begin{eqnarray*}
&&\scal{(u\ministuffle_\phi v)\ministuffle_\phi w}{t}=\scal{(u\ministuffle_\phi v)\otimes w}{\Delta_{\ministuffle_\phi}(t)}\cr
&=&\scal{u\otimes v\otimes w}{(\overline{\Delta_{\ministuffle_\phi}\otimes I})\Delta_{\ministuffle_\phi}(t)}\cr
&=& \scal{u\otimes v\otimes w}{(\overline{I\otimes \Delta_{\ministuffle_\phi}})\Delta_{\ministuffle_\phi}(t)}\cr
&=&\scal{u\otimes  (v\ministuffle_\phi w)}{\Delta_{\ministuffle_\phi}(t)}=
\scal{u\ministuffle_\phi (v\ministuffle_\phi w)}{t}
\end{eqnarray*}
which proves the associatvity of the law $\ministuffle_\phi$. Conversely, if $\ministuffle_\phi$ is associative, the direct expansion of the right hand side of 
\begin{equation}
0=(y_i\ministuffle_\phi y_j)\ministuffle_\phi y_k-
y_i\ministuffle_\phi (y_j\ministuffle_\phi y_k)
\end{equation}
proves the associativity of $\phi$.\\
iii) We suppose that $(\gamma_{x,y}^z)_{x,y,z\in X}$ satisfies \mref{cond_D}. In this case $\Delta_{\ministuffle_\phi}$ takes its values in $\ncp{A}{Y}\otimes \ncp{A}{Y}$ so its dual, the law $\ministuffle_\phi$ is dualizable. Conversely, if $Im(\Delta_{\ministuffle_\phi})\subset \ncp{A}{Y}\otimes \ncp{A}{Y}$, one has, for every 
$s\in I$
$$
\sum_{n,m\in I} \gamma_{n,m}^s\, y_n\otimes y_m=\Delta(y_s)-(y_s\otimes 1+1\otimes y_s)
\in \ncp{A}{Y}\otimes \ncp{A}{Y}
$$ 
which proves the claim.  
\end{proof}
{\bf From now on, we suppose that $\phi : AY\otimes AY\longrightarrow AY$
is an associative and commutative law (of algebra) on $AY$.}
\begin{theorem}\label{phishuffstruct}
Let $A$ be a $\Q$-algebra. Then if $\phi$ is dualizable
\footnote{For the pairing defined by 
$$
(\forall x,y\in Y)(\scal{x}{y}=\delta_{x,y})\ .
$$
}, let $\Delta_{\ministuffle_\phi} : \ncp{A}{Y}\longrightarrow \ncp{A}{Y}\otimes \ncp{A}{Y}$ denote its dual comultiplication, then
\begin{itemize}
\item[a)] $\mathcal{B}_\phi=(\AY, {\tt conc}, 1_{Y^*}, \Delta_{\ministuffle_\phi},\varepsilon) $ is a bialgebra.\\
\item[b)] If $A$ is a $\Q$-algebra then, the following conditions are equivalent\\
i) $\mathcal{B}_\phi$ is an enveloping bialgebra\\
ii) the algebra $AY$ admits an increasing  filtration $\Big((AY)_n\Big)_{n\in \N}$
\begin{eqnarray*}
	(AY)_0=\{0\}\subset (AY)_1\subset\cdots\subset (AY)_n\subset (AY)_{n+1}\subset\cdots
\end{eqnarray*}
compatible with both the multiplication and the comultiplication $\Delta_{\ministuffle_\phi}$ {\it i.e.}
\begin{eqnarray*}
(AY)_p(AY)_q &\subset& (AY)_{p+q}\cr
	\Delta_{\ministuffle_\phi}((AY)_n)&\subset& \sum_{p+q=n} (AY)_p\otimes (AY)_q\ .
\end{eqnarray*}    
iii) $\B_\phi$ is isomorphic to $(\ncp{A}{Y},\mathtt{conc},1_{Y^*},\Delta_{\shuffle},\ep)$ as a bialgebra.\\
iv) $I^+$ is $\star$-nilpotent.
\end{itemize}
\end{theorem}

\begin{proof} We only prove the following implication (the other ones are easy)\\ 
$iv)\Longrightarrow iii)$ Let us set $y'_s=\pi_1(y_s)$, then using a rearrangement of the star-log of the diagonal series, we have 
\begin{equation}
y_s=\sum_{k\ge1}\frac1{k!}\sum_{s'_1+\cdots +s'_k=s}\pi_1(y_{s'_1})\ldots\pi_1(y_{s'_k})
\end{equation}
This proves that the multiplicative morphism given by $\Phi(y_s)=y'_s$ is an isomorphism. But this morphism is such that $\Delta_{\ministuffle_\phi} \circ \Phi=(\Phi\ot \Phi)\circ \Delta_{\shuffle}$ which proves the claim.\\
\end{proof}
\begin{remark}
i) Theorem \ref{phishuffstruct} a) holds for general (dualizable, coassociative) $\phi$ be it commutative of not.\\ 
ii) It can happen that there is no antipode (and then, $I^+$ cannot be $\star$-nilpotent) as the following example shows.\\ 
Let $Y=\{y_0,y_1\}$ and $\phi(y_i,y_j)=y_{(i+j\ mod\ 2)}$, then
\begin{eqnarray}\label{exzee2}
&&\Delta(y_0)=y_0\ot 1+1\ot y_0+ y_0\ot y_0+ y_1\ot y_1\cr
&&\Delta(y_1)=y_1\ot 1+1\ot y_1+ y_0\ot y_1+ y_1\ot y_0
\end{eqnarray}
then, from eqns \ref{exzee2}, one derives that $1+y_0+y_1$ is group-like. As this element has no inverse in $\ncp{K}{Y}$. Thus, the bialgebra $\mathcal{B}_\phi$ cannot be a Hopf algebra. 
\\
iii) When $I^+$ is nilpotent, the antipode exists and is computed by 
\begin{equation}\label{antipodeI+}
	a_{{\stuffle_\phi}}=(I)^{*-1}=(e+I^+)^{*-1}=\sum_{n\geq 0}(-1)^k(I^+)^{*k}
\end{equation}
(see section (\ref{CCQMM})).\\
iv) In QFT, the antipode of a vector $h\in \B$ is computed by
\begin{equation}
	S(1)=1,\ S(h)=-h+\sum_{(1)(2)}S(h_{(1)})h_{(2)} 
\end{equation}
and by using the fact that $S$ is an antimorphism. This formula is used in contexts where $I^+$ is $\star$-nilpotent (although the concerned bialgebras are often not cocommutative).  
Here, one can prove this recursion from \mref{antipodeI+}.  
\end{remark}
\section{Conclusion}
We have depicted the framework which is common to different kinds of shuffles. For all these, provided that $I_+$ be $*$-nilpotent, the bialgebra 
$$
(\AY, {\tt conc}, 1_{Y^*}, \Delta_{\ministuffle_\phi},\varepsilon)
$$
is isomorphic to 
$$
(\AY, {\tt conc}, 1_{Y^*}, \Delta_{\minishuffle},\varepsilon)
$$
and the straightening algorithm is simply the morphism which sends each $y_s\in Y$ to $\pi_1(y_s)=\log(I)(y_s)$ (this bialgebra is then a Hopf algebra). In other cases, such as the infiltration given by 
$$
\Delta(y_s)=y_s\ot 1+1\ot y_s+y_s\ot y_s
$$ 
group-like elements without inverse may appear (and therefore no Hopf structure can be hoped).\\


\begin{thebibliography}{99}
%
\bibitem{berstel_reutenauer}{Berstel J., Reutenauer C.}.--
Rational series and their languages, Springer (1988).
%
\bibitem{B_E}{Bourbaki, N.} .--
{N. Bourbaki, Th\'eorie des ensembles, Springer (2006)}
%
\bibitem{B_alg_I_III} {Boubaki, N.} .--
{N. Bourbaki, Alg\`ebre, Chap I-III, Springer (2006)}
%
\bibitem{B_Lie_II_III} {Boubaki, N.} .--
{N. Bourbaki, Groupes et Alg\`ebres de Lie, Chap II-III, Springer (2006)}
%
\bibitem{BDM}{V.C. Bui, G. H. E. Duchamp, Hoang Ngoc Minh}.--
\textit{Sch\"utzenberger's factorization on the (completed) Hopf algebra of $q-$stuffle product},
Journal of Algebra, Number Theory and Applications (2013), {\bf 30}, No. 2 , pp 191 - 215.
%
\bibitem{BDKMT}{V.C. Bui, G. H. E. Duchamp, Ladji Hane, Hoang Ngoc Minh, C. Tollu}.--
\textit{Dual bases for noncommutative symmetric and quasi-symmetric functions via monoidal factorization},
(in preparation).
%
\bibitem{orsay}{J.Y.Enjalbert, Hoang Ngoc Minh}.--
\textit{Combinatorial study of Hurwitz colored poly\-z\^etas},
Discrete Mathematics, 1. {\bf 24} no. 312 (2012), p. 3489-3497.
%
\bibitem{luque} Duchamp G.H.E. , Flouret M., Laugerotte \'E., Luque J.-G., {\it Direct and dual laws for automata with multiplicities}, Theoretical Computer Science {\bf 267} (2001) 105-120. 
%
\bibitem{loday1} Loday J.-L., {\it S\'erie de Hausdorff, idempotents Eul\'eriens et 
alg\`ebres de Hopf}, Expositiones Mathematicae {\bf 12} (1994) 165-178. 
%
\bibitem{lyndon}{Chen K.T., Fox R.H., Lyndon R.C.}.--
{Free differential calculus, IV. The quotient groups of the lower central series},
 Ann. of Math. , {\bf 68} (1958) pp. 81-95, 
%
\bibitem{DDHS} Deneufch\^atel M., Duchamp G. H. E., Hoang Ngoc Minh, Solomon A. I., {\it Independence of hyperlogarithms over function fields via algebraic combinatorics}, Lecture Notes in Computer Science (2011), {\bf 6742} (2011), 127-139.
arXiv:1101.4497v1 [math.CO] 
%
\bibitem{SLC62} Duchamp G.H.E., Tollu C., Penson K.A. and Koshevoy G.A., {\it Deformations of Algebras: Twisting and Perturbations}, S\'eminaire Lotharingien de Combinatoire, B62e (2010). 
%
\bibitem{SLC43}{Hoang Ngoc Minh, Jacob~G., Oussous N.E., Petitot M.}.--
    {\it Aspects combinatoires des polylogarithmes et des sommes d'Euler-Zagier},
    \textit{S\'eminaire Lotharingien de Combinatoire}, B43e (2000).
%
\bibitem{SLC44}{Hoang Ngoc Minh, Jacob~G., Oussous N.E., Petitot M.}.--
    {\it De l'alg\`ebre des $\zeta$ de Riemann multivari\'ees \`a
    l'alg\`ebre des $\zeta$ de Hurwitz multivari\'ees},
    \textit{S\'eminaire Lotharingien de Combinatoire}, {\bf 44} (2001).
%
\bibitem{FPSAC97}{Hoang Ngoc Minh \& Petitot M.}.--
    {\it Lyndon words, polylogarithmic functions
    and the Riemann $\zeta$ function},
    Discrete Math., {\bf 217}, 2000, pp. 273-292.
%
\bibitem{FPSAC98}{Hoang Ngoc Minh, Petitot M. and  Van der Hoeven J.}.--
   {\it Polylogarithms and Shuffle Algebra},
   Proceedings of FPSAC'98, 1998.
%
\bibitem{FPSAC99}{Hoang Ngoc Minh, Petitot, M. and Van der Hoeven J.}.--
   {\it L'alg\`ebre des polylogarithmes par les s\'eries g\'en\'eratrices},
   Proceedings of FPSAC'99, 1999.
%
\bibitem{acta}{Hoang Ngoc Minh}.--
\textit{On a conjecture by Pierre Cartier about a group of associators},
Acta Math. Vietnamica (2013),  {\bf 38}, Issue 3, pp 339-398.
%
\bibitem{VJM}{Hoang Ngoc Minh}.--
\textit{Structure of polyzetas and Lyndon words}, 
Vietnamese Math. J. (2013),  {\bf 41}, Issue 4, pp 409-450.
%
\bibitem{hoffman}{Hoffman, M.E.} {\it The algebra of multiple harmonic series}, J. of Alg., {\bf 194} (1997) 477-495.
%
\bibitem{radford}{Radford D.E.}.--
{\it A natural ring basis for shuffle algebra and an application to group schemes}, 
Journal of Algebra, {\bf 58} (1979) 432-454.
%
\bibitem{ree}{Ree R.},-- {\it Lie elements and an algebra associated with shuffles}
	 Ann. of Math. {\bf 68} (1958) 210--220.
%
\bibitem{reutenauer}{Reutenauer, C.}.--
\textit{Free Lie Algebras}, London Math. Soc. Monographs,
        New Series-{\bf 7}, Oxford University Press, 1993.
%
\bibitem{th_ung} J.-Y. Thibon, B.-C.-V. Ung, {\it Quantum quasi-symmetric functions and Hecke algebras}, Journal of Physics A {\bf 29} (1996), 7337-7348
%
\bibitem{viennot}{Viennot, G}.--
Bases des alg\`ebres de Lie Libres et factorisations des mono\"\i des libres, Lect. Notes in Math., {\bf 691}, Springer Verlag Berlin, 1978, 124p.
%
\bibitem{zagier}{D. Zagier}.--
    {\it Values of zeta functions and their applications},
    in ``First European Congress of Mathematics'',
    {\bf 2}, Birkh\"auser (1994), 497-512.
\end{thebibliography}
\end{document}